\documentclass{mcom-l}

\newtheorem{theorem}{Theorem}[section]
\newtheorem{lemma}[theorem]{Lemma}

\theoremstyle{definition}

\theoremstyle{remark}
\newtheorem{remark}[theorem]{Remark}

\numberwithin{equation}{section}

\usepackage{amsmath}
\usepackage{amssymb}
\usepackage{latexsym}
\usepackage{wrapfig}
\usepackage{graphicx}

\newcommand{\R}{\mathbb R}

\newcommand{\bD}{\mathbf D}

\newcommand{\bH}{\mathbf H}
\newcommand{\bI}{\mathbf I}

\newcommand{\bP}{\mathbf P}

\newcommand{\bn}{\mathbf n}
\newcommand{\bp}{\mathbf p}

\newcommand{\bs}{\mathbf s}

\newcommand{\bx}{\mathbf x}

\newcommand{\T}{\mathcal T}

\newcommand{\cT}{\mathcal F}

\newcommand{\uDelta}{{\Delta}_{\Gamma}}
\newcommand{\unabla}{{\nabla}_{\Gamma}}

\newcommand{\rd}{\mathrm{d}}
\newcommand{\la}{\left\langle}
\newcommand{\ra}{\right\rangle}

\def\dO{{\partial\Omega} }
\def\div{\operatorname{div} }

\begin{document}

\title[Unfitted  FEM for PDEs on surfaces]{A narrow-band unfitted  finite element method for elliptic PDEs posed on surfaces}
\author{Maxim A. Olshanskii}
\author{Danil Safin}

\address{Department of Mathematics, University of Houston, Houston, Texas 77204-3008} 
\email{(molshan,dksafin)@math.uh.edu}
\thanks{This work has been supported  by  NSF through the Division of Mathematical Sciences grant 1315993}

\begin{abstract}
The paper studies a method for solving elliptic partial differential equations
posed on hypersurfaces in $\mathbb{R}^N$, $N=2,3$. The method allows a surface to be given implicitly as a zero level of a level set function.  A surface equation is extended to a narrow-band neighborhood of the surface. The resulting extended equation is a non-degenerate PDE and it is solved
on a bulk mesh that is unaligned to the surface. An unfitted finite element method is used to discretize extended equations.   Error estimates are proved for finite element solutions in the bulk domain and restricted to the surface. The analysis admits finite elements of a higher order and gives sufficient conditions for
archiving the optimal convergence order in the energy norm. Several numerical examples illustrate the properties of the method.
 \end{abstract}


\maketitle


\section{Introduction}
Partial differential equations posed on surfaces arise in mathematical models for many natural phenomena: diffusion along grain boundaries \cite{grain2}, lipid interactions in biomembranes \cite{membrains1}, and transport of surfactants on multiphase flow interfaces \cite{GrossReuskenBook}, as well as in many engineering and bioscience applications: vector field visualization \cite{vector}, textures synthesis \cite{texture1}, brain warping \cite{imaging2}, fluids in lungs \cite{lungs} among others.
Thus, recently there has been a significant increase of interest  in developing and analyzing numerical methods for  PDEs on surfaces.

One natural approach to solving PDEs on surfaces numerically is based on surface triangulation.  In this class of methods, one typically assumes that a parametrization of a surface is given and the surface is approximated by a family of consistent  regular triangulations. It is common to assume that all nodes of the triangulations lie on the surface. The analysis of a finite element method based on surface triangulations was first done in~\cite{Dziuk88}. To avoid surface triangulation and remeshing (if the surface  evolves), another approach was taken in \cite{BCOS01}: It was proposed to extend a partial differential equation from the surface to a set of positive Lebesgue measure in $\R^N$. The resulting PDE is then solved in one dimension higher, but can be solved on a mesh that is unaligned to the surface.
A surface is allowed  to be defined implicitly as a zero set of  a given level set function. However, the resulting bulk elliptic or parabolic equations are degenerate, with no diffusion acting  in the direction normal to the surface. A version of the method, where only an $h$-narrow band around the surface is used to define a finite element method, was studied in  \cite{DDEH}. A fairly complete overview of finite element methods for surface PDEs and more references can be found in the recent review paper~\cite{DE2013}.

For an elliptic equation on a compact hypersurface, the present  paper considers a new extended non-degenerate formulation, which is uniformly elliptic in a bulk domain containing the surface. We analyze a Galerkin finite element method for solving the extended equation. The bulk domain is allowed to be a narrow band around the surface with width proportional to a mesh size. Thus the number of degrees of freedom used in computations stays asymptotically optimal, when the mesh size decreases. The finite element method we apply here is unfitted: The mesh does not respect the surface or the boundary of the narrow band. This property is important from the practical point of view. No parametrization of the surface is required by the method. The surface can be given implicitly and the implementation requires only an approximation of its distance function. We analyse the approximation properties of the method and prove  error estimates for finite element solutions in the bulk domain and restricted to the surface. The analysis allows finite elements of higher order and gives sufficient conditions for archiving optimal convergence order in the energy norm. We remark that up to date the analysis of higher order finite element methods for surface PDEs is largely an open problem: In \cite{Demlow09} a higher order extension of the  method from \cite{Dziuk88} was analysed under the assumption that a parametrization of $\Gamma$ is known.
The analysis of a coupled surface-bulk problem from \cite{ER2013} also admits a higher order discretization by isoparametric finite elements on a triangulation fitted to a given surface.

Another unfitted finite element method for elliptic equations posed on surfaces was introduced in~\cite{ORG09,OlsR2009}. That method does not use an extension of the surface partial differential equation.  It is instead based on a restriction (trace) of the outer  finite element spaces to a surface. We do not compare these two different approaches in the paper. 

The remainder of the paper is organized as follows.  Section~\ref{s_prel} collects some necessary definitions and preliminary results. In section~\ref{s_form}, we recall the extended PDE approach from \cite{BCOS01} and
 introduce a different non-degenerate extended formulation. In section~\ref{s_FEM}, we consider a finite element method. Finite element method error analysis is presented in section~\ref{s_error}. Section~\ref{s_numer} shows the result of several numerical experiments. Finally, section~\ref{s_concl} collects some closing remarks.

\section{Preliminaries}\label{s_prel}
We assume that $\Omega$ is an open subset in $\R^N$, $N=2,3$ and $\Gamma$ is a connected $C^2$ compact
hypersurface contained in $\Omega$.
For a sufficiently smooth function $g: \Omega \to \R$ the tangential
gradient (along $\Gamma$) is defined by
\begin{equation*}
  \unabla g=\nabla g - (\bn_\Gamma\cdot\nabla g) \, \bn_\Gamma,
\end{equation*}
where $\bn_\Gamma$ is the outward normal vector on $\Gamma$.
By $\uDelta$ we denote  the Laplace--Beltrami operator on
$\Gamma$, $\uDelta=\nabla_\Gamma\cdot\nabla_\Gamma$.

This paper deals with elliptic  equations posed on $\Gamma$. As a model problem, we
consider the  Laplace--Beltrami  problem:
\begin{equation}\label{LBeq}
-\uDelta u + \alpha\, u =f\quad\text{on}~\Gamma,
\end{equation}
with some strictly positive $\alpha\in L^\infty(\Gamma)$.
The corresponding weak form of \eqref{LBeq} reads: For given $f\in L^2(\Gamma)$  determine $u \in H^1(\Gamma)$ such that
\begin{equation}\label{weak}
\int_\Gamma \unabla u \unabla v+\alpha\,uv\,\rd\bs= \int_\Gamma f v\, \rd\bs\qquad\text{for all}~~v\in H^1(\Gamma).
\end{equation}
The solution $u$ to \eqref{weak} is unique and  satisfies $u\in H^2(\Gamma)$, with $\|u\|_{H^2(\Gamma)}\le
c\|f\|_{L^2(\Gamma)}$ and a constant $c$ independent of $f$, cf.~\cite{Dziuk88}.
\medskip

Further we consider a surface embedded in $\R^3$, i.e. $N=3$. With obvious  minor modifications all results hold if $\Gamma$ is a curve in $\R^2$.
Denote by $\Omega_d$ a domain consisting of all points within a distance from $\Gamma$ less than some $d>0$:
\begin{equation}\label{Omega_d}
\Omega_d = \{\, \bx \in  \R^3~:~{\rm dist}(\bx,\Gamma) < d\, \}.
\end{equation}
Let $\phi: \Omega_d \rightarrow \R$ be the
signed distance function, $|\phi(x)|:={\rm dist}(\bx,\Gamma)$ for all
$\bx \in \Omega_d$. The surface $\Gamma$ is the zero level
set of $\phi$:
\begin{equation*}
\Gamma=\{\bx\in\mathbb{R}^3\,:\,\phi(\bx)=0\}.
\end{equation*}
We may assume $\phi < 0$ on the interior of $\Gamma$  and $\phi >0$ on the exterior.
We define $\bn(\bx):=\nabla \phi(\bx)$ for all
$\bx \in \Omega_d$. Thus, $\bn=\bn_{\Gamma}$ on $\Gamma$,  and $|\bn (\bx)|=1$ for all $\bx\in \Omega_d$.  The Hessian of $\phi$ is denoted by $\bH$:
\begin{equation*}
  \bH(\bx)=\mathrm{D}^2\phi(\bx) \in \R^{3 \times 3} \quad \text{for all} ~~\bx \in \Omega_d.
\end{equation*}
The eigenvalues of $\bH(\bx)$ are  $\kappa_1(\bx),
\kappa_2(\bx)$, and 0. For  $\bx \in \Gamma$, the eigenvalues $\kappa_i(\bx)$, $i=1,2$, are
the principal curvatures.

We need the orthogonal projector
\[
 \bP(\bx)= \bI-\bn(\bx)\otimes\bn(\bx) \quad \text{for all}~~\bx \in \Omega_d.
\]
Note that the tangential gradient can be written as $\unabla g(\bx)= \bP \nabla g(\bx)$ for
$\bx \in \Gamma$.
We introduce a locally
orthogonal coordinate system by using the  projection $\bp:\, \Omega_d \rightarrow
\Gamma$:
\begin{equation}\label{loc_proj}
 \bp(\bx)=\bx-\phi(\bx)\bn(\bx) \quad \text{for all}~~\bx \in \Omega_d.
\end{equation}
Assume that $d$ is sufficiently small such that the decomposition $\bx=\bp(\bx)+ \phi(\bx)\bn(\bx)$ is unique for all $\bx \in \Omega_d$.
We shall use an extension operator  defined as follows. For a  function $v$ on $\Gamma$
we define
\begin{equation} \label{extension}
 v^e(\bx):= v(\bp(\bx)) \quad \text{for all}~~\bx \in \Omega_d.
\end{equation}
Thus, $ v^e$ is the extension of $v$ along normals on $\Gamma$, it satisfies $\bn\cdot \nabla v^e=0$ in $\Omega_d$, i.e., $v^e$ is constant along normals to $\Gamma$.
Computing the gradient of $v^e(\bx)$ and using \eqref{loc_proj} and \eqref{extension} gives
\begin{equation}
\nabla  v^e(\bx)  = (\mathbf{I} -  \phi(\bx)\bH(\bx))\unabla v(\bp(\bx))\quad \text{for}~~\bx \in\Omega_{d}.
 \label{nabla1}
\end{equation}
For higher order derivatives, assume the surface is sufficiently smooth $\Gamma\in C^{k+1}$, $k=2,3,\dots$.
This yields $\phi\in C^{k+1}(\Omega_d)$, see \cite{DistFunc}, and hence $\bp(\bx)\in [C^k(\Omega_d)]^3$. Differentiating  \eqref{extension} gives for a sufficiently smooth $v$
\begin{equation}
|D^\alpha v^e(\bx)|  \le c\,\sum_{l=1}^k \sum_{|\mu|=l} |D^\mu_\Gamma v(\bp(\bx))|\quad \text{for}~~\bx \in\Omega_{d},\quad |\alpha|=k,
 \label{Dk}
\end{equation}
where a constant $c$ can be taken independent of $\bx$ and $v$.

From (2.5) in \cite{DD07} we have the following formula for the eigenvalues of $\bH$:
\begin{equation}\label{curv}
\kappa_i(\bx)= \frac{\kappa_i(\bp(\bx))}{1 + \phi(\bx)\kappa_i(\bp(\bx))}\quad \text{for} ~\bx \in \Omega_d.
\end{equation}
Since $\Gamma\in C^2$ and $\Gamma$ is compact, the principle curvatures of $\Gamma$ are uniformly bounded and $d$ can be taken sufficiently small to satisfy
 \begin{equation}\label{ass_d}
d\le \Big(\,2\max_{\bx\in\Gamma}(|\kappa_1(\bx)|+|\kappa_2(\bx)|)\,\Big)^{-1}.
\end{equation}
For such choice of $d$, we obtain using \eqref{curv}
 \begin{equation}\label{ass1}
 |\phi(\bx)|=\mbox{dist}(\bx,\Gamma)\le d\le \frac12\|\bH(\bx)\|^{-1}_2\quad\text{for}~\bx\in\Omega_d.
\end{equation}
 The inequality \eqref{ass1} yields the bounds for the spectrum and the determinant of the symmetric matrix $\bI-\phi\bH$:
 \begin{equation}\label{spec}
 \mbox{sp}(\bI-\phi\bH)\in\left[\mbox{$\frac12,\frac32$}\right],\quad \mbox{$\frac14$}\le\mbox{det}(\bI-\phi\bH)\le\mbox{$\frac94$}\quad
 \text{in}~\Omega_d.
 \end{equation}
Therefore, the matrix $(\bI-\phi\bH)^{-1}$ is well defined and its norm is uniformly bounded in $\Omega_d$.

\section{Extensions of the surface PDEs} \label{s_form}
In this section, we review some well-known results for numerical methods based on surface PDEs extensions and  define a suitable  extension of the surface equation \eqref{LBeq} to a neighborhood of~$\Gamma$.

\subsection{Review of results}
In \cite{BCOS01} Bertalmio et al. suggested to extend a PDE off a surface to every level set
 of the indicator function $\phi$ in some neighborhood of $\Gamma$. Applied to \eqref{LBeq} this leads to the
problem posed in $\Omega_d$:
 \begin{equation}
\label{1.2}
-|\nabla\phi|^{-1}\div|\nabla\phi| \bP\nabla u+ \alpha^e\,u=f^e\quad  \hbox{ in } \Omega_d.
\end{equation}
The  corresponding weak  formulation of \eqref{1.2} was shown to be well-posed in \cite{Burger}.
The weak solution is sought  in the anisotropic space
\[
H_P=\{v\in L^2(\Omega_d)\,:\,\bP\nabla v\in (L^2(\Omega_d))^{3}\}.
\]
On every level set of $\phi$ the solution to \eqref{1.2}  does not depend
on a data in a neighborhood of this level set.  Indeed, the diffusion in \eqref{1.2} acts only in the direction
  tangential to level sets of $\phi$ and one may consider \eqref{1.2} as a collection of
of independent surface problems posed on every level set. Hence, the surface  equation \eqref{LBeq} is embedded in \eqref{1.2} and if a smooth solution to \eqref{1.2} exists, then
restricted to $\Gamma$ it  solves the original Laplace-Beltrami  problem \eqref{LBeq}. With no ambiguity,
we shall denote by $u$ both the solutions to surface and extended problems.

The major numerical advantage of any extended formulation is that one may apply standard
discretization methods to solve  equations in the volume domain $\Omega_d$ and further take the trace of  computed solutions on $\Gamma$ (or on a approximation of $\Gamma$). Computational experiments from \cite{BCOS01,Burger,Greer06,XuZhao} suggest that these traces of numerical solutions  are reasonably good approximations to the solution of the surface problem  \eqref{LBeq}.

Numerical analysis of surface equations discretization methods based on extensions  is by far not completed: Error estimates for finite element methods for \eqref{1.2} were shown in~\cite{Burger,DDEH}.
Error estimate in \cite{Burger} was established in the integral volume norm
\[
\|v\|_{H_P}^2:=\|v\|^2_{L^2(\Omega_d)}+\|\bP\nabla v\|^2_{L^2(\Omega_d)},
\]
rather than in a surface norm for $\Gamma$.
In \cite{DDEH}, a finite element method based on triangulations not fitted to the curvilinear boundary of $\Omega_d$ was
studied.  The first order convergence was proved in the surface $H^1$ norm, if the band width $d$ in \eqref{Omega_d} is of the order of mesh size and if a quasi-uniform triangulation of $\Omega$ is assumed. For linear elements this estimate is of the optimal order in energy norm.

The extended formulation \eqref{1.2} is numerically convenient, but has a number of potential issues, as noted already in   \cite{BCOS01} and reviewed in \cite{DDEH,Greer06}. No boundary conditions are needed for \eqref{1.2}, if the boundary of the bulk domain $\Omega_d$ consists of level sets of $\phi$. However some auxiliary boundary conditions are often required by a numerical method. The extended equation \eqref{1.2} is defined in a domain in one dimension higher than the surface equation. This leads to involving extra degrees of freedom in computations. If $\Omega_d$ is a narrow band around $\Gamma$, then handling numerical boundary conditions may effect the quality of the discrete solution.
Finally, the  second order term in the extended formulation \eqref{1.2} is degenerate, since no diffusion acts in the direction normal to level sets of $\phi$. The current understanding of   numerical methods for degenerate elliptic and parabolic equations is still limited.

An  improvement to the original extension of surface PDEs was introduced by  Greer in~\cite{Greer06}.
Greer suggested to use the  non-orthogonal scaled projection operator
\begin{equation}\label{tildeP}
\widetilde{\bP}:=(\bI-\phi \bH)^{-1} \bP
\end{equation}
on  tangential planes of the level sets of $\phi$ instead of $\bP$.  For a smooth $\Gamma$, one can always consider  small enough $d>0$ such that $\widetilde{\bP}$  is well defined in $\Omega_d$ .  If $\phi$ is the singed distance function and all data ($\alpha$ and $f$ for equations \eqref{1.2}) is extended to the neighborhood of $\Gamma$  according to \eqref{extension}, i.e. constant along normals, then  one can easily show (see \cite{Greer06,ChOlsh}) that
the solution to the new extended equation is constant in normal directions:
\begin{equation}\label{const}
(\bn\cdot\nabla) u=0\quad\text{a.e. in}~\Omega_d.
\end{equation}
The property \eqref{const} is crucial, since it allows to add diffusion in the normal direction without altering solution. Doing this, one obtains a non-degenerated elliptic operator.
Thus, for solving the heat equation on a surface,   it was suggested in \cite{Greer06} to include the additional term
$
-c^2_n\div(\bn\otimes\bn)\nabla u
$
 in the extended formulation  with a coefficient  $c^2_n$. For the planar case, $\Omega_d\in\mathbb{R}^2$, the recommendation was to set $c_n=(1-\phi\kappa_0)$,  $\kappa_0=\kappa(p(\bx))$, $\kappa$ is the curvature of $\Gamma$ ($\Gamma$ is a curve in the planar case).

\subsection{Non-degenerate extended equations}
Here we deduce another extension of \eqref{LBeq}: Let $\phi$ be the signed distance function and $\mu=\mbox{det}(\bI-\phi\bH)$,  $\alpha^e$ and $f^e$ are the normal extensions to $\alpha$ and $f$. We look for $u$ solving the following elliptic problem
 \begin{equation}
\label{ExtNew}
\begin{split}
-\div\mu(\bI-\phi\bH)^{-2}\nabla u+\alpha^e\mu\, u&=f^e\mu\quad \text{in}~~\Omega_d\\
\frac{\partial u}{\partial \bn}&=0\qquad \text{on}~~\dO_d.
\end{split}
\end{equation}
The Neumann boundary condition in \eqref{ExtNew} is the natural boundary condition. To see this, note the identity $\bH\bn=0$  and  that $\bn$ coincides (up to a sign) with a normal vector on the boundary of $\Omega_d$.
 Hence, one has $(\bI-\phi\bH)^{-1}\bn=\bn$ and for a sufficiently smooth $u$:
\[
((\bI-\phi\bH)^{-2}\nabla u) \cdot \bn= (\nabla u) \cdot ((\bI-\phi\bH)^{-2}\bn)=(\nabla u)\cdot\bn=\frac{\partial u}{\partial \bn}=0\quad \text{on}~~\dO_d.
\]
The weak formulation of \eqref{ExtNew} reads:  Find $u\in H^1(\Omega_d)$ satisfying
\begin{equation}\label{weak_bulk}
\int_{\Omega_d}\left[(\bI-\phi\bH)^{-2}\nabla u\cdot \nabla v + \alpha^e\,uv\right]\mu\,d\bx=\int_{\Omega_d}f^ev\mu\,d\bx
\quad \hbox{for all } v\in H^1(\Omega_d).
\end{equation}
Thanks to \eqref{spec} the corresponding bilinear form
 \[
 a(u,v):=\int_{\Omega_d}\left[(\bI-\phi\bH)^{-2}\nabla u\cdot \nabla v + \alpha^e\,uv\right]\mu\,d\bx
 \]
 is continuous and coercive on $H^1(\Omega_d)$.

 The next theorem states several  results about the well-posedness of \eqref{ExtNew} and its relation to the
 surface equations \eqref{LBeq}.
\begin{theorem}\label{Th0}
Assume $\Gamma\in C^2$, $\alpha\in L^\infty(\Gamma)$, $f\in L^2(\Gamma)$. The following assertions hold:\\[-2ex]
\begin{description}
\item{}(i) The problem~\eqref{ExtNew} has the unique weak solution
$u\in H^1(\Omega_d)$, which satisfies  $\|u\|_{H^1(\Omega_d)}\le C\,\|f^e\|_{L^2(\Omega_d)}$,
with a constant $C$ depending only on $\alpha$ and $\Gamma$;\\[-2ex]
\item{}(ii) For the solution $u$ to \eqref{ExtNew} the trace function $u|_{\Gamma}$ is
an element of $H^1(\Gamma)$ and solves the weak formulation of the surface equation \eqref{weak}.\\[-2ex]
\item{}(iii) The solution $u$ to \eqref{ExtNew} satisfies \eqref{const}. Using the notion of normal extension, this can be written as
$
u = (u|_\Gamma)^e~ \text{in}~\Omega_d;
$\\[-2ex]
\item{}(iv)  Additionally assume $\Gamma\in C^3$, then $u\in H^{2}(\Omega_d)$ and
$
\|u\|_{H^2(\Omega_d)}\le C\,\|f^e\|_{L^2(\Omega_d)},
$
with a constant $C$ depending only on  $\alpha$, $\Gamma$ and $d$;\\[-2ex]
\end{description}
\end{theorem}
\begin{proof} Since the bilinear form $a(u,v)$ is elliptic and continuous
in $H^1(\Omega_d)$, the Lax-Milgram lemma implies the result in (i). Assumption $\Gamma\in C^3$ yields $\phi\in C^3(\Omega_d)$, see \cite{DistFunc}, and hence $\bH,\mu\in C^1(\Omega_d)$ and $\dO_d\in C^3$. The regularity theory for elliptic PDEs with Neumann boundary data~\cite{Grisvard} implies the result in item (iv).

Now we are going to show how the bulk equation \eqref{ExtNew}  relates to the surface equation \eqref{LBeq}.
For $r\in(-d,d)$, denote by $\Gamma_r$ the level set surface on distance $r$ from~$\Gamma$:
\[
\Gamma_r=\{\bx\in\Omega_d~:~~ \phi(\bx)=r\}.
\]
Since $\phi$ is the sign distance function, the coarea formula gives
\begin{equation}\label{L3_aux1}
\int_{\Omega_d} f\,d\bx= \int_{-d}^d\int_{\Gamma_r} f\,d\bs\,dr\quad\mbox{for}~f\in L^1(\Omega_d).
\end{equation}
For area elements on $\Gamma_r$ and $\Gamma$ we have
\begin{equation}\label{elements}
\mu(\bx)\, d\bs(\bx)=\mbox{det}(\bI-\phi(\bx)\bH(\bx))\, d\bs(\bx)=d\bs(\bp(\bx))\quad\mbox{for}~\bx\in\Gamma_r.
\end{equation}

Denote by $u\in H^1(\Gamma)$ the unique  solution to the surface equations \eqref{LBeq}. Recall that $u^e\in H^1(\Omega_d)$ denotes the normal extension of $u$. From the weak formulation of the surface equation \eqref{weak} and transformation formulae \eqref{nabla1} and \eqref{elements} we infer
\[
\int_{\Gamma_r} \left[(\bI-\phi\bH)^{-2}\nabla u^e \nabla v^e+\alpha^e\,u^ev^e\right]\,\mu\rd\bs= \int_{\Gamma_r} f^e v^e\,\mu \rd\bs\quad\text{for any}~~v\in H^1(\Gamma).
\]
Since $\nabla v^e=\bP\nabla v^e$ and $\bP\nabla v$ is the tangential gradient which depends only on
values of $v$ on $\Gamma_r$, but not on an extension, we can rewrite the above identity as
\begin{equation}\label{Th0_aux1}
\int_{\Gamma_r} \left[(\bI-\phi\bH)^{-2}\nabla u^e \bP\nabla v+\alpha^e\,u^ev\right]\,\mu\rd\bs= \int_{\Gamma_r} f^e v\,\mu \rd\bs\quad\text{for any}~~v\in H^1(\Gamma_r).
\end{equation}
Assuming $v$ is a smooth function on $\Omega_d$, and so $v|_{\Gamma_r}\in H^1(\Gamma_r)$,
we can integrate \eqref{Th0_aux1} over all level sets for $r\in(-d,d)$ and apply the coarea formula
\eqref{L3_aux1} to obtain
\begin{equation}\label{Th0_aux2}
\int_{\Omega_d} \left[(\bI-\phi\bH)^{-2}\nabla u^e \bP\nabla v+\alpha^e\,u^ev\right]\,\mu\rd\bs= \int_{\Omega_d} f^e v\,\mu \rd\bs\quad\text{for any}~~v\in C^\infty(\Omega_d).
\end{equation}
Now we use $\bP=\bP^T$ and  $\bH\bP=\bP\bH\,\Rightarrow\,\bP(\bI-\phi\bH)^{-2}=(\bI-\phi\bH)^{-2}\bP$
to get from \eqref{Th0_aux2}
\begin{equation*}
\int_{\Omega_d} \left[(\bI-\phi\bH)^{-2}\nabla u^e \nabla v+\alpha^e\,u^ev\right]\,\mu\rd\bs= \int_{\Omega_d} f^e v\,\mu \rd\bs\quad\text{for any}~~v\in C^\infty(\Omega_d).
\end{equation*}
Applying the density argument we conclude that the normal extension of the
surface solution  $u^e$ solves the weak formulation \eqref{weak_bulk} of the bulk problem \eqref{ExtNew}.
Since the solution to \eqref{ExtNew} is unique, we have proved assertions (ii) and (iii) of the theorem.
\end{proof}

The formulation \eqref{ExtNew} has the following advantages over \eqref{1.2}:
The equation  \eqref{ExtNew} is non-degenerate and uniformly elliptic, the extended problem has no parameters to be defined,  the boundary conditions are given and consistent with \eqref{const}.
One theoretical advantage of the  formulation  \eqref{ExtNew} over \eqref{1.2} is that  the Agmon-Douglis-Nirenberg regularity theory  is readily  applicable if the data is smooth.

We remark that the volumetric formulation of surface equations can be easily extended for the case of anisotropic surface diffusion.
Indeed, let  $\bD(\bx)\in\left(L^{\infty}(\Gamma)\right)^{3\times3}$ be a symmetric positive definite tensor acting in tangential subspaces of $\Gamma$, i.e. $\bD\bn=0$ on $\Gamma$.
Consider the surface diffusion equation:
\begin{equation*}
-\mbox{div}_\Gamma \bD\nabla u + \alpha\, u =f\quad\text{on}~\Gamma.
\end{equation*}
Thanks to $\bD=\bD^T,~\bD\bn=0~\Rightarrow~\bP\bD^e=\bD^e\bP=\bD^e$, repeating the same arguments as in the proof of Theorem~\ref{Th0} leads to the extended problem:
 \begin{equation*}
\begin{split}
-\div\mu\widetilde{\bD}\nabla u+\alpha^e\mu\, u&=f^e\mu\quad \text{in}~~\Omega_d\\
\frac{\partial u}{\partial \bn}&=0\qquad \text{on}~~\dO_d,
\end{split}
\end{equation*}
with $\widetilde{\bD}=(\bI-\phi\bH)^{-1}[\bD^e+d^e\bn\otimes\bn](\bI-\phi\bH)^{-1}$, $\bD^e$ is the componentwise normal extension of $\bD$ and  $d\in L^{\infty}(\Gamma)$ is arbitrary positive on $\Gamma$. A reasonable choice of $d$ can be
the minimizer of the K-condition number\footnote{The definition of the K-condition number of a symmetric positive definite matrix $A\in\R^{n\times n}$ is $K(A)=\frac{(\mbox{tr}(A)/n)^n}{\mbox{det(A)}}$, see~\cite{Axelsson}.} of the volume diffusion tensor $\bD^e+d^e\bn\otimes\bn$ on $\Gamma$. One finds  $d=\frac1{N-1}\mbox{tr}(\bD)$, where $N=2,3$ is the outer space dimension. Note that the isotropic diffusion
problem \eqref{LBeq} fits this more general case with $\bD=\bP$.
Including anisotropic surface diffusion tensor would not bring any additional difficulty to the analysis below.
However, for the sake of brevity we consider further only isotropic diffusion.

\section{Finite element method}\label{s_FEM}
Let $\Gamma\in C^2$, $\Gamma\subset\Omega$, where $\Omega\subset\mathbb{R}^3$ is a polyhedral domain.
Assume we a given a family $\{\T_h\}_{h>0}$ of regular triangulations of $\Omega$ such that $\max_{T\in\T_h}\mbox{diam}(T) \le h$.
For a tetrahedron  $T$ denote by $\rho(T)$  the diameter of the inscribed ball.
Denote
\begin{equation}\label{beta}
\beta=\sup_{T\in\T_h}\mbox{diam}(T)/\inf_{T\in\T_h}\rho(T)\,.
\end{equation}
For the sake of analysis, we assume that triangulations of $\Omega$ are quasi-uniform, i.e., $\beta$ is uniformly bounded in $h$.
The band width $d$ satisfies \eqref{ass1} and such that $\Omega_d\subset\Omega$.

It is computationally convenient  not to align (not to fit) the mesh  to $\Gamma$ or  $\dO_d$.
Thus, the computational domain $\Omega_h$ will be a narrow band containing $\Gamma$ with a piecewise smooth boundary which is not fitted to
the mesh $\T_h$. 

 Let $\phi_h$ be a continuous piecewise smooth, with respect to $\T_h$, approximation  of the surface distance function. Assume $\phi_h$ is known  and  satisfies
\begin{equation}\label{phi_h}
\|\phi-\phi_h\|_{L^\infty(\Omega)}+ h \|\nabla(\phi-\phi_h)\|_{L^\infty(\Omega)}\le c\,h^{q+1}
\end{equation}
with some $q\ge1$. Then one defines
\begin{equation}\label{Omega_h}
\Omega_h=\{\, \bx \in  \R^3~:~|\phi_h(\bx)| < d\, \}.
\end{equation}
Note, that in some applications the surface $\Gamma$ may not be known explicitly and only a finite element  approximation $\phi_h$ to the distance function $\phi$ is known. Otherwise, one may set $\phi_h:=I_h(\phi)$, where  $I_h$ is a suitable
piecewise polynomial interpolation operator. If $\phi_h$ is a $P_1$ continuous finite element function, then
$\Omega_h$ has a piecewise planar boundary. In this practically convenient case, \eqref{phi_h} is assumed with $q=1$.

Alternatively,  for $\Gamma$ given explicitly one  may build a piecewise planar approximation to $\dO_d$ as suggested in \cite{BarrettElliott}.
We briefly recall it here. 
Assume $d\ge h$ (relaxing this assumption is possible, but requires additional technical  considerations).
Consider all tetrahedra that have vertices   both inside $\Omega_d$, $\phi(\bx_i) < d$, and outside, $\phi(\bx_j) > d$.
Let $p_{1}$ be a intersection point of $\Gamma$ with the edge $\overline{\bx_i\bx_j}$. For any tetrahedra there can be
three or four  such points $p_1,...,p_k$.  Inside each such tetrahedron,  $\partial\Omega_d$ is approximated by either by the plane $(p_1,p_2,p_3)$, if $k=3$, or by two pieces of planes $(p_1,p_2,p_3)$ and $(p_2,p_3,p_4)$.


Denote by $\T_d$ the set of all tetrahedra having nonempty intersection with $\Omega_h$:
\[
\T_d=\bigcup\limits_{T\in\T_h} \{T\,:~T\cap\Omega_h\neq\emptyset\}.
\]
We always assume that $\Gamma\subset\Omega_h\subset\T_d\subset\Omega_{d'}\subset\Omega$, with some $d'\le c\,d$ satisfying \eqref{ass_d}.

The space of all continuous piecewise polynomial functions of a degree $r\ge1$ with respect to   $\T_d$ is our finite element space:
\begin{equation}\label{FEspace}
V_h:=\{v\in C(\T_d)\,:~ v|_T\in P_r(T)\quad\forall\,T\in\T_d\},\quad r\ge1.
\end{equation}
The finite element method reads: Find $u_h\in V_h$ satisfying
\begin{equation}\label{FEmeth}
\int_{\Omega_h}\left[(\bI-\phi_h\bH_h)^{-2}\nabla u_h\cdot\nabla v_h + \alpha^e\,u_h v_h\right]\,\mu_hd\bx =\int_{\Omega_h} f^ev_h\,\mu_hd\bx\quad\forall\,v_h\in V_h.
\end{equation}
This is the method we analyse further  in this paper.

If $\Gamma$ is given explicitly, one can compute $\phi$ and $\bH$ and set  $\phi_h=\phi$, $\bH_h=\bH$ and $\mu_h=\mbox{det}(\bI-\phi_h\bH_h)$ in \eqref{FEmeth}.
Otherwise, if  the surface $\Gamma$ is known approximately as, for example, the zero level set of a finite element distance function $\phi_h$, then, in general, $\phi_h\neq \phi$  and one has to define a discrete Hessian $\bH_h\approx \bH$ and also set $\mu_h=\mbox{det}(\bI-\phi_h\bH_h)$. A discrete Hessian $\bH_h$  can be obtained from $\phi_h$ by a  recovery method, see, e.g.,~\cite{Hessian1,Hessian0}.
At this point, we assume that some $\bH_h$ is provided and denote by $p\ge 0$ the approximation order for $\bH_h$ in the (scaled) $L^2$-norm:
\begin{equation} \label{Hh}
|\Omega_h|^{-\frac12}\|\bH-\bH_h\|_{L^2(\Omega_h)}\le c h^p,
\end{equation}
where $|\Omega_h|$ denotes the volume of $\Omega_h$.

\begin{remark}\rm
From the implementation viewpoint, it is most convenient to use polyhedral (polygonal) computation domains $\Omega_h$, which
 corresponds to the second order approximation of $\dO_d$ ($q=1$ in \eqref{phi_h}).
It appears that in this case, the optimal order convergence result with $P_1$ finite elements  in narrow-band domains, $d=O(h)$, holds already for $p=0$ in \eqref{Hh}, e.g. \textit{$\bH_h=0$ is the suitable choice}.
This follows from the error analysis below and supported by the results of numerical experiments in Section~\ref{s_numer}.
\end{remark}

Finally, we assume  that $\phi_h$ and $\bH_h$ satisfy condition~\eqref{ass1}, which is a reasonable assumption once $d$ and $h$ are sufficiently small.  Hence  the $h$-dependent bilinear form
\[
a_h(u_h,v_h)=\int_{\Omega_h}\left[(\bI-\phi_h\bH_h)^{-2}\nabla u_h\cdot\nabla v_h + \alpha^e\,u_h v_h\right]\,\mu_hd\bx
\]
is continuous and elliptic uniformly in $h$.

\section{Error analysis}\label{s_error}

If $\phi_h=\phi$ and $\bH_h=\bH$, then \eqref{FEmeth} is closely related to the unfitted finite element method from \cite{BarrettElliott} for
an elliptic equation with Neumann boundary conditions. However, applied to \eqref{FEmeth} the analysis of \cite{BarrettElliott} does not account for the anisotropy of computational domain and leads to suboptimal convergence results in surface norms.
Therefore, to prove an optimal order convergence in the $H^1(\Gamma)$ norm, we use a different  framework, which also allows to cover the case $\phi_h\neq \phi$, $\bH_h\neq \bH$ and higher order finite elements.

We need a further mild  assumption on how well the mesh resolves the geometry.
Since $\Gamma\subset \Omega_h$ the boundary of $\Omega_h$ is decomposed into two disjoint sets, $\dO_h=\dO_h^+\cup\dO_h^-$,
such that $\phi>0$ on $\dO_h^+$ and $\phi<0$ on $\dO_h^-$.
 We assume that $\dO_h^{+}$ is a graph of a function $\eta_+(\bx)$, $\bx\in\Gamma$, in the local coordinates induced by the
  projection \eqref{loc_proj}. The same is assumed for $\dO_h^{-}$ and $\eta_-(\bx)$, $\bx\in\Gamma$.

To estimate the consistency error of the finite element method, we need results in the next two lemmas.

\begin{lemma}\label{L_Phi} Consider $\Omega_h$ as defined in \eqref{Omega_h} for some $\phi_h$ satisfying \eqref{phi_h}. For sufficiently small $h$, there exists a one-to-one mapping $\Phi_h:\,\Omega_h\to\Omega_d$, $\Phi_h\in \left(W^{1,\infty}(\Omega_h)\right)^3$ such that
\begin{equation}\label{mapping}
\|\mbox{\rm id}-\Phi_h\|_{L^\infty(\Omega_h)}+ h \|\bI-\mathrm{D}\Phi_h\|_{L^\infty(\Omega_h)}\le c\,h^{q+1},
\end{equation}
where $\mathrm{D}\Phi_h$ is the Jacobian matrix. Moreover, the mapping $\Phi_h$ is such that $\bp(\Phi_h(\bx))=\bp(\bx)$ for any $\bx\in\Omega_h$.
\end{lemma}
\begin{proof}
Since $\Gamma\in C^2$ and $\dO_h$ is piecewise smooth, we have $\eta_{\pm}\in W^{1,\infty}(\Gamma)$.
Further in the proof we consider $\eta_+$. Same conclusions would be true for $\eta_-$.

Consider $\Omega_h$ as defined in \eqref{Omega_h}, then   $\eta_+(\bx)$ is an implicit function given by
\begin{equation}\label{implicit}
\phi_h(\bx+\eta_+(\bx)\bn(\bx))=d,\quad \bx\in\Gamma.
\end{equation}
For the distance function it holds $\phi(\bx+\alpha\bn(\bx))=\alpha$, for $\bx\in\Gamma$ and $\alpha\in[-d,d]$. Hence  from \eqref{implicit} and \eqref{phi_h} we conclude
\begin{equation}\label{Est1}
|\eta_+(\bx)-d|=|\phi(\bx+\eta_+(\bx)\bn(\bx))-\phi_h(\bx+\eta_+(\bx)\bn(\bx))|\le c\,h^{q+1},\quad
\bx\in\Gamma.
\end{equation}
To compute the surface gradient of $\eta_+$,  we differentiate    \eqref{implicit} and find using the chain rule:
\[
\unabla \eta_+(\bx) = -\frac{(I+\eta_+(\bx)\bH(\bx))\unabla \phi_h(\bx')}{\nabla \phi_h(\bx')\cdot \nabla\phi(\bx)},\quad
\bx'=\bx+\bn\eta_+,\quad \bx\in\Gamma.
\]
Noting $\unabla \phi=0$, $\nabla\phi(\bx)=\nabla\phi(\bx')$, and using \eqref{phi_h} we estimate  for sufficiently small mesh size $h$
\[
\begin{split}
|\unabla \eta_+(\bx)|&=\left|\frac{(I+\eta_+(\bx)\bH(\bx))\unabla ( \phi_h(\bx')-\phi(\bx'))}
{\nabla \phi_h(\bx')\cdot \nabla\phi(\bx')}\right|\\
&=2\left|\frac{(I+\eta_+(\bx)\bH(\bx))\unabla ( \phi_h(\bx')-\phi(\bx'))}
{|\nabla \phi_h(\bx')|^2+|\nabla\phi(\bx')|^2-|\nabla (\phi_h(\bx')-\phi(\bx'))|^2}\right|\\
&\le
c\left|\frac{(I+\eta_+(\bx)\bH(\bx))h^q }
{|\nabla \phi_h(\bx')|^2+1-ch^{2q}}\right|
 \le ch^q \quad \mbox{a.e.~on}~\Gamma.
\end{split}
\]
From this and \eqref{Est1} we infer
\begin{equation}\label{eta_h}
\|\eta_+-d\|_{L^\infty(\Gamma)}+ h \|\unabla\eta_+\|_{L^\infty(\Gamma)}\le c\,h^2.
\end{equation}

Now the required mapping can be defined as
\[
\Phi_h(\bx)=\left\{\begin{aligned}
\bx-\bn(\bx)\frac{\left(\eta_+^e(\bx)-d\right)\phi(\bx)}{\eta_+^e(\bx)}&~~\text{if}~~ \phi(\bx)\ge0\\
\bx-\bn(\bx)\frac{\left(\eta_-^e(\bx)+d\right)\phi(\bx)}{\eta_-^e(\bx)}&~~\text{if}~~ \phi(\bx)<0\\
\end{aligned}
\right.\quad\bx\in\Omega_h.
\]
The property  $\bp(\Phi_h(\bx))=\bp(\bx)$  is obviously satisfied by the construction of $\Phi_h$.
Due to  the triangle inequality and \eqref{Est1} we have
\[
d\le|\eta_+(\bx)|+ |d-\eta_+(\bx)|\le |\eta_+(\bx)| + ch^{q+1}\le |\eta_+(\bx)| + cdh^{q}\quad
\bx\in\Gamma
\]
Therefore, for sufficiently small $h$ there exists a mesh independent constant $c>0$ such that $|\eta_+^e(\bx)|\ge cd\ge c |\phi(\bx)|$.
Hence the estimate for $|\bx-\Phi_h(\bx)|$  follows from  \eqref{eta_h}.
The estimate for $\|\bI-\mathrm{D}\Phi_h\|_2$ also follows from \eqref{eta_h} with the help of \eqref{nabla1}.\\
\end{proof}


\begin{lemma}\label{L_AB}
For two symmetric positive definite  matrices $A,B\in\R^{N\times N}$, assume $\|A-B\|_2\le\delta$, where $\|\cdot\|_2$ denotes the spectral matrix norm. Then it holds
\begin{align}\label{AB1}
\|A^2-B^2\|_2&\le \delta \|A+B\|_2,\\
\|A^{-1}-B^{-1}\|_2&\le \delta \|B^{-1}\|_2\|A^{-1}\|_2, \label{AB2}\\
|\mbox{\rm det}(A)-\mbox{\rm det}(B)|&\le \delta N \max\{\|A\|_2^{N-1},\|B\|_2^{N-1}\}. \label{AB3}
\end{align}
\end{lemma}
\begin{proof}
For completeness, we give the proof of  these elementary results.
For a symmetric matrix $A\in\R^{N\times N}$ we have  $\|A\|_2=\sup\limits_{0\neq x\in\R^3}|\la Ax,x\ra|/|x|^2$.
Hence the estimate \eqref{AB1} follows from
\begin{multline*}
|\la(A^2-B^2)x,x\ra|=|\la(A-B)x,(A+B)x\ra|\\ \le |(A-B)x| |(A+B)x|\le \|A-B\|_2 \|A+B\|_2|x|^2.
\end{multline*}
We write $A\le B$ if the matrix $B-A$ is positive semidefinite and recall that for two  symmetric  positive definite
matrices $A\le B$ yields $B^{-1}\le A^{-1}$. Using this and that
$\|A-B\|_2\le\delta$ is equivalent to $ -\delta I\le A-B\le \delta I$ we obtain
\[
A\le \delta I +B\quad\Rightarrow\quad A\le (\delta\lambda_{\min}^{-1}(B)+1)B\quad\Leftrightarrow\quad
B^{-1}\le (\delta\lambda_{\min}^{-1}(B)+1)A^{-1}
\]
This implies
\[B^{-1}-A^{-1}\le \delta\lambda_{\min}^{-1}(B)A^{-1}\le \delta\lambda_{\min}^{-1}(B)\lambda_{\max}(A^{-1})I=
\delta \|B^{-1}\|_2\|A^{-1}\|_2I
\]
Same arguments show $ A^{-1}-B^{-1}\le  \delta \|B^{-1}\|_2\|A^{-1}\|_2I$.

To prove \eqref{AB3}, note that $\mbox{det}(A)=\prod_{k=1}^N\lambda_k(A)$ for eigenvalues $0<\lambda_1(A)\le\dots\le\lambda_N(A)$. Hence
\[
|\mbox{det}(A)-\mbox{det}(B)|\le N \max\limits_{k=1,\dots,N}|\lambda_k(A)-\lambda_k(B)|\max\{\lambda_N^{N-1}(A),\lambda_N^{N-1}(B)\}.
\]
The Courant--Fischer theorem gives for the $k$th eigenvalue of
a symmetric matrix the characterization
\[
\lambda_k(A)=\max_{S\in{{\mathcal{V}}_{k-1}}} \min_{0\neq y\in S^\perp}
\frac{\langle A y , y \rangle}{|y|^2}\,,
\]
where $\mathcal{V}_{k-1}$ denotes the family of all $(k-1)$-dimensional
subspaces of $\R^{N}$.
The inequality $\min_y(a(y)+b(y))\le\min_ya(y)+\max_yb(y)$ implies that $\min_ya(y)-\min_yb(y) \leq
\max_y(a(y)-b(y))$.  Using this we estimate
\[
 \lambda_k(A)-\lambda_k(B) \le
\max_{S\in{\mathcal{V}_{k-1}}} \max_{y\in S^\perp}
\frac{\langle(A-B) y ,y\rangle}{|y|^2} \le
\max_{y\in\mathbb{R}^{N}}\frac{\langle (A-B) y ,y\rangle}{|y|^2} \le
\|{A}-B\|_2\le\delta.
\]
One can estimate the difference $\lambda_k(B)-\lambda_k(A)$ in the same way.\\
\end{proof}

Now we are prepared for the error analysis of our finite element method.
First we prove an estimate for the error in a volume norm.

\begin{theorem}\label{Conv1} Let $\Gamma\in C^3$. Assume $u^e$ and $u_h$ solve
problems \eqref{ExtNew} and \eqref{FEmeth}, respectively, and $u^e\in W^{1,\infty}(\Omega_d)$, 
$f\in L^\infty(\Gamma)$.
Then it holds
\begin{equation}\label{err_estH1}
\|u^e-u_h\|_{H^1(\Omega_h)} \le C \left(\inf_{v_h\in V_h}\|u^e-v_h\|_{H^1(\Omega_h)}+ d^{\frac12}h^q+d^\frac32\,h^p\right),
\end{equation}
where  $q$ and $p$ are defined in \eqref{phi_h} and \eqref{Hh}, respectively.
\end{theorem}
\begin{proof} Since $u^e$ is constant along normals, we can consider a normal extension of $u^e$ on $\Omega_{d'}$. Then the bilinear form $a_h(u^e,v_h)$ is well defined and we can apply the second Strang's lemma. Hence, to show \eqref{err_estH1}, we need to check the bound
\begin{equation}\label{Th1_aux1}
\frac{|a_h(u^e,v_h)-\int_{\Omega_h}f^ev_h\mu_h|}{\|v_h\|_{H^1(\Omega_h)}}\le
 C\,(d^{\frac12}h^q+d^\frac32h^p).
\end{equation}
We introduce the auxiliary bilinear form
 \[
 a^h(u,v):=\int_{\Omega_h}\left[(\bI-\phi\bH)^{-2}\nabla u\cdot \nabla v + \alpha^e\,uv\right]\,\mu d\bx,\quad \mbox{for}~u,v\in H^1(\Omega_h).
 \]
It holds
 \begin{equation}\label{Th1_aux2}
 \begin{split}
|a^h(u^e,&v_h)-a_h(u^e,v_h)|  \\
&\le \|u^e\|_{W^{1,\infty}(\Omega_h)}\|v_h\|_{H^1(\Omega_h)}\\
&\qquad\times\left(\int_{\Omega_h}\|\mu(\bI-\phi\bH)^{-2}
-\mu_h(\bI-\phi_h\bH_h)^{-2}\|^2_2 + |\alpha^e| |\mu-\mu_h|^2 d\bx\right)^{\frac12}.
\end{split}
\end{equation}
Recall that matrices $(\bI-\phi\bH)$ and $(\bI-\phi_h\bH_h)$ are symmetric positive definite
and $\mu=\mbox{det}(\bI-\phi\bH)$,  $\mu_h=\mbox{det}(\bI-\phi_h\bH_h)$. Since $\phi,\bH$ and $\phi_h, \bH_h$
both satisfy \eqref{ass1}, for the spectrum and determinants of $(\bI-\phi\bH)$ and $(\bI-\phi_h\bH_h)$ the bounds in \eqref{spec} hold
uniformly in $\bx$ and $h$.  We use this and Lemma~\ref{L_AB}  to estimate
\[
\begin{aligned}
\|\mu(\bI-\phi\bH)^{-2}&
-\mu_h(\bI-\phi_h\bH_h)^{-2}\|_2+  |\alpha^e||\mu-\mu_h|\\ &\le C\,\|\phi\bH-\phi_h\bH_h\|_2
\le C(|\phi|\|\bH-\bH_h\|_2+\|\bH_h\|_2|\phi-\phi_h|)\\ &\le C(d\,\|\bH-\bH_h\|_2+h^{q+1}).
\end{aligned}
\]
Applying \eqref{Hh} we get from \eqref{Th1_aux2}
 \begin{equation}\label{Th1_aux3}
\begin{aligned}
|a^h(u^e,v_h)-a_h(u^e,v_h)|  &\le C|\Omega_h|^{\frac12}(d\,h^p+h^{q+1})\|u^e\|_{W^{1,\infty}(\Omega_h)}\|\nabla v_h\|_{L^2(\Omega_h)}\\
&\le C\,d^{\frac12}\,(d h^p+h^{q+1})\|v_h\|_{H^1(\Omega_h)}.
\end{aligned}
\end{equation}

It remains to  estimate $|a^h(u^e,v_h)-\int_{\Omega_h}f^ev_h\mu_h|$.
Following \cite{DDEH}  we consider $v_h\circ\Phi_h^{-1}\in H^1(\Omega_d)$ as a test function in \eqref{weak}. By the triangle inequality we have
\begin{multline}\label{Th1_aux4}
|a^h(u^e,v_h)-\int_{\Omega_h}f^ev_h\mu_h|\\ \le |a^h(u^e,v_h)-a(u^e,v_h\circ\Phi_h^{-1})|+\left|\int_{\Omega_d}f^ev_h\circ\Phi_h^{-1}\mu-\int_{\Omega_h}f^ev_h\mu_h\right|.
\end{multline}
Using the integrals transformation rule and the identities
\[
\begin{aligned}
\nabla(v_h\circ\Phi_h^{-1})&= (\mathrm{D}\Phi_h)^{-T}(\nabla v_h)\circ\Phi_h^{-1},\\
(\nabla u^e)\circ\Phi_h &= (\mathrm{D}\Phi_h)^{-T}\nabla(u^e\circ\Phi_h)=(\mathrm{D}\Phi_h)^{-T}\nabla u^e,\\
\alpha^e\circ\Phi_h&=\alpha^e,\quad u^e\circ\Phi_h=u^e,
\end{aligned}
\]
 we calculate
\begin{multline*}
a(u^e,v_h\circ\Phi_h^{-1})=\int_{\Omega_d}\left[(\bI-\phi\bH)^{-2}\nabla u^e\cdot \nabla(v_h\circ\Phi_h^{-1}) + \alpha^e\,u^e(v_h\circ\Phi_h^{-1})\right]\,\mu d\bx\\
=\int_{\Omega_h}\left[(\bI-\phi\bH)^{-2}\circ\Phi_h \left((\mathrm{D}\Phi_h)^{-T} \nabla u^e\right) \cdot \left((\mathrm{D}\Phi_h)^{-T}\nabla v_h\right)\right.\\\left. + \alpha^e\,u^e v_h\right]|\mbox{det}(\mathrm{D}\Phi_h)|\mu\circ\Phi_h d\bx.
\end{multline*}
Thus, we have
\begin{equation}\label{Th1_aux5}
a(u^e,v_h\circ\Phi_h^{-1})-a^h(u^e,v_h)=\int_{\Omega_h} R_1\nabla u^e \cdot\nabla v_h+R_2 u^e v_h\,d\bx,
\end{equation}
with
\begin{equation}\label{Th1_aux6}
\begin{split}
\|R_1\|_2&
=\|(\mathrm{D}\Phi_h)^{-1} \mu(\bI-\phi\bH)^{-2}\circ\Phi_h(\mathrm{D}\Phi_h)^{-T}|\mbox{det}(\mathrm{D}\Phi_h)|-
\mu(\bI-\phi\bH)^{-2}\|_2\\
&\le C\|\mu(\bI-\phi\bH)^{-2}-\mu(\bI-\phi\bH)^{-2}\circ\Phi_h\|_2
\\ &\qquad+
\|(\mathrm{D}\Phi_h)^{-1}\mu(\bI-\phi\bH)^{-2}(\mathrm{D}\Phi_h)^{-T}|\mbox{det}(\mathrm{D}\Phi_h)|-\mu(\bI-\phi\bH)^{-2}\|_2\\
&\le C\,\sup_{\Omega_{d'}}\|\nabla(\mu(\bI-\phi\bH)^{-2})\|_F\|\mbox{id}-\Phi_h\|_2
\\ &\qquad+
C\,\|\mu(\bI-\phi\bH)^{-2}\|_2
\|(\mathrm{D}\Phi_h)^{-1}-\bI\|_2\, |\mbox{det}(\mathrm{D}\Phi_h)|-1|\\
 &\le c h^q.
\end{split}
\end{equation}
and
\begin{equation}\label{Th1_aux7}
|R_2|=|\alpha^e(|\mbox{det}(\mathrm{D}\Phi_h)|\mu\circ\Phi_h-\mu)|\le ch^q\|\alpha^e\|_{L^\infty(\Omega_h)}\le c\,h^q\|\alpha\|_{L^\infty(\Gamma)}\le c\,h^q.
\end{equation}
The notion $\|\cdot\|_F$ was used above for the Frobenius norm of a tensor. The term $\|\nabla(\mu(\bI-\phi\bH)^{-2})\|_F$ is uniformly bounded on $\Omega_{d'}$ thanks to the
assumption $\Gamma\in C^3 \Rightarrow  \phi\in C^3(\Omega_{d'})\Rightarrow \mu, \bH \in C^1(\Omega_{d'})$.
From \eqref{Th1_aux5}--\eqref{Th1_aux7} we obtain
\begin{equation}\label{Th1_aux8}
\begin{split}
|a(u^e,v_h\circ\Phi_h^{-1})-a^h(u^e,v_h)|&\le c\,h^q |\Omega_h|^{\frac12}\|u^e\|_{W^{1,\infty}(\Omega_h)}\|v_h\|_{H^1(\Omega_h)}\\
&\le c\,d^{\frac12}\,h^q\|v_h\|_{H^1(\Omega_h)}.
\end{split}
\end{equation}

Since $f^e\circ\Phi_h=f^e$, we also have
\[
\int_{\Omega_d}f^ev_h\circ\Phi_h^{-1}\mu\,d\bx=\int_{\Omega_h}f^ev_h|\mbox{det}(\mathrm{D}\Phi_h)|\mu\circ\Phi_hd\bx.
\]
Hence
\begin{equation*}
\begin{split}
\left|\int_{\Omega_d}f^ev_h\circ\Phi_h^{-1}\mu-\int_{\Omega_h}f^ev_h\mu_h\right|& \le c(h^q \|f^e\|_{L^2(\Omega_h)}+h^pd^{\frac32}\|f^e\|_{L^\infty(\Omega_h)})\|v_h\|_{L^2(\Omega_h)}\\
&\le c(h^q \|f^e\|_{L^2(\Omega_h)}+h^pd^{\frac32}\|f^e\|_{L^\infty(\Omega_h)})\|v_h\|_{L^2(\Omega_h)}.
\end{split}
\end{equation*}
Since for the normal extension it holds $\|f^e\|_{L^2(\Omega_{d'})}\le c(d')^{\frac12}\|f\|_{L^\infty(\Gamma)}$, and $d'\le c\,d$, we obtain
\begin{equation}\label{Th1_aux9}
\left|\int_{\Omega_d}f^ev_h\circ\Phi_h^{-1}\mu-\int_{\Omega_h}f^ev_h\mu_h\right|
\le c(d^{\frac12}\,h^q+d^{\frac32}\,h^p)\|v_h\|_{L^2(\Omega_h)}.
\end{equation}
Combining estimates in \eqref{Th1_aux3}, \eqref{Th1_aux8}, \eqref{Th1_aux9} we prove \eqref{Th1_aux1}.\\
\end{proof}

\begin{remark}\rm Note that the extra regularity assumption $u^e\in W^{1,\infty}(\Omega_d)$ was only used
to estimate the consistency error in \eqref{Th1_aux2} due to the Hessian approximation and \eqref{Th1_aux8}.
If we alternatively assume the Hessian $O(h^p)$ approximation order in the stronger norm $L^\infty(\Omega_h)$, then
it is sufficient to let $u\in H^{1}(\Gamma)$ and employ the estimate
$\|u^e\|_{H^1(\Omega_{d'})}\le c(d')^{\frac12}\|u\|_{H^1(\Gamma)}$ in \eqref{Th1_aux2}  and \eqref{Th1_aux8}.
The same remark is valid for the statement of Theorem~\ref{ThMain} below.
\end{remark}

Now we turn to proving the error estimate in the surface $H^1$-norm.
The result of the lemma below follows from Lemma~3 in \cite{Hansbo}, see also Lemma~4.4 in \cite{ChOlsh}.

\begin{lemma}\label{L_trace}
Let $T\in \T_h$. Denote  $\widetilde{K}=T\cap\Gamma$, then for any $v\in H^1(T)$ it holds
\begin{equation}\label{eq_trace}
\|v\|_{L^2(\widetilde{K})}^2\le C\,(h^{-1}\|v\|_{L^2(T)}^2+h\|\nabla v\|_{L^2(T)}^2),
\end{equation}
where the constant $C$ may depend only on $\Gamma$ and the minimal angle condition for $\T_h$.
\end{lemma}
\smallskip

Now we prove our main result concerning the convergence of the finite element method \eqref{FEmeth}.

\begin{theorem}\label{ThMain} Let $\Gamma\in C^{r+2}$, $d\le c\,h$, $f\in L^\infty(\Gamma)$, and assume $u\in W^{1,\infty}(\Gamma)\cap H^{r+1}(\Gamma)$ solves
the surface problems \eqref{LBeq} and $u_h\in V_h$ solves \eqref{FEmeth}. Then it holds
\[
\|u-u_h\|_{H^1(\Gamma)} \le C\,(h^r+h^{p+1}+h^q),
\]
where a constant $C$ is independent of $h$, and $r\ge1$, $p\ge0$, $q\ge1$ are the finite elements, Hessian recovery,
 and distance function approximation orders defined in \eqref{FEspace}, \eqref{phi_h} and \eqref{Hh}, respectively.
\end{theorem}
\begin{proof} Since $\Gamma\in C^3$, the regularity $u\in W^{1,\infty}(\Gamma)\cup H^2(\Gamma)$ implies for the normal extension:  $u^e\in W^{1,\infty}(\Omega_d)\cup H^2(\Omega_d)$. Hence the assumptions of  Theorem~\ref{Conv1} are satisfied.

We apply estimate  \eqref{eq_trace} componentwise to $v=\nabla(u^e-u_h)$. This leads to the bound
\begin{equation}\label{Th2_aux1}
\|\nabla(u^e-u_h)\|_{L^2(\widetilde{K})}^2\le C\,(h^{-1}|u^e-u_h|_{H^1(T)}^2+h|u^e-u_h|_{H^2(T)}^2),
\end{equation}
Denote by $I_h u^e$ the Lagrange interpolant for $u^e$ on $\T_d\subset\Omega_{d'}$. Thanks to the inverse
inequality and approximation properties of finite elements we have
\[
\begin{split}
|u^e-u_h|_{H^2(T)}&\le |u^e-I_h u^e|_{H^2(T)} + |I_h u^e-u_h|_{H^2(T)}\\
 &\le C(h^{r-1}|u^e|_{H^{r+1}(T)}+h^{-1} |I_h u^e-u_h|_{H^1(T)})\\ &\le
C(h^{r-1}|u^e|_{H^{r+1}(T)}+h^{-1} (|u^e-I_h u^e|_{H^1(T)}+|u^e-u_h|_{H^1(T)}) \\
&\le
C(h^{r-1}|u^e|_{H^{r+1}(T)}+h^{-1} |u^e-u_h|_{H^1(T)}).
\end{split}
\]
Substituting this estimate to \eqref{Th2_aux1} and summing up over all elements from $\T_h$ with non-empty intersection with $\Gamma$ and using   $|\unabla(u-u_h)|=|\unabla(u^e-u_h)|\le|\nabla(u^e-u_h)|$ on $\Gamma$, we get
 \begin{equation*}
\begin{aligned}
\|\unabla(u-u_h)\|_{L^2(\Gamma)}^2&\le C\,\sum_{\scriptsize  \begin{array}{c} T\in\T_\Gamma\\T\cap\Gamma\neq\emptyset\end{array}}\left(\,h^{-1}\|\nabla(u^e-u_h)\|_{L^2(T)}^2
+h^{2r-1}|u^e|_{H^{r+1}(T)}^2\right)\\
&\le C\,(h^{-1}\|\nabla(u^e-u_h)\|_{L^2(\Omega_h)}^2+h^{2r-1}|u^e|_{H^{r+1}(\Omega_{d'})}^2).
\end{aligned}
\end{equation*}
To estimate the first term on the righthand side, we apply the volume error estimate from Theorem~\ref{Conv1}, a standard approximation result
for finite element functions from $V_h$ and recall $d=O(h)$.  This leads to
\begin{equation}
 \label{Th2_aux2}
\|\unabla(u-u_h)\|_{L^2(\Gamma)}^2\le  C\,(h^{2q}+h^{2p+2}+h^{2r-1}|u^e|_{H^{r+1}(\Omega_{d'})}^2).
\end{equation}
Finally, integrating \eqref{Dk} for $k=r+1$ over $\Omega_{d'}$ and repeating arguments of Lemma~3.2 in \cite{ORG09}
we find
\begin{equation}
 \label{Th2_aux3}
|u^e|_{H^{r+1}(\Omega_{d'})}^2\le C\,d'\,\|u\|_{H^{r+1}(\Gamma)}^2.
\end{equation}
Estimate \eqref{Th2_aux2}, \eqref{Th2_aux3} and $d'\le ch$ yield
\begin{equation*}
\|\unabla(u-u_h)\|_{L^2(\Gamma)}^2\le  C\,(h^{2q}+h^{2p+2}+h^{2r}).
\end{equation*}

To show an estimate for the surface $L^2$-norm of the error, we apply the estimate \eqref{eq_trace} for $v=u^e-u_h$
and proceed with similar arguments.
\end{proof}

\section{Numerical examples}\label{s_numer}
In this section, we present results of several numerical experiments. They illustrate
the performance of the method and the analysis of the paper. In all experiments the band width, $d=\gamma h$, is ruled by the parameter $\gamma$ and always stays proportional to the mesh width. Results of a few experiments with a fixed mesh-independent band  width and fitted meshes can be found in \cite{ChOlsh}.

If $\Omega_h$ is a polyhedral domain (the approximation order equals $q=1$ in \eqref{phi_h}), then the implementation of the method is straightforward and this is what we use in all numerical examples. In this case, already $P_1$ finite elements deliver optimal convergence results. The technical difficulty of using higher order approximations of $\Omega$ is the need to define a suitable numerical integration rule  over a part of tetrahedra $T\in\cT_d$ bounded by a zero level set of $\phi_h$, where $\phi_h$ is a polynomial of degree $\ge2$ on $T$. In this paper for higher order elements, we use an implementation where sufficiently many quadrature nodes are taken within each cut triangle to guarantee accurate enough integration.  Integrating over arbitrary  cut element in $O(1)$ (asymptotically optimal for $h\to0$ number of operation) is a non-standard task (see, i.g., a recent paper~\cite{Integr}) and
we address it in a separate paper, currently in preparation.

\medskip

\noindent{\it Experiment~1}.
We start with the example of the  Laplace--Beltrami problem \eqref{LBeq}
on a unit circle in $\mathbb{R}^2$ with a known solution so that we are able to calculate
the error between the continuous and discrete solutions. We set $\alpha=1$ and consider
\[
u(r,\phi)=\cos(5\phi)
\]
in polar coordinates, similar to the Example~5.1 from~\cite{DDEH}.

We perform a regular  uniform triangulation of $\Omega=(-2,2)^2$ and  $h=2^{-\ell}\times 10^{-1}$ denotes further a maximal
edge length of triangles for the refinement level $\ell$. Thus the grid is not aligned with $\dO_d$. We use piecewise affine continuous finite elements, $r=1$, and $\Omega_h$ is a polygonal approximation of $\Omega_d$ as described in \cite{BarrettElliott}, $q=1$. Convergence results in $H^1(\Gamma)$ and $L^2(\Gamma)$ norms
are shown in Tables~\ref{tab1a} and~\ref{tab1b} for the choices $\bH_h=\bH$ and $\bH_h=0$, respectively.
Error reduction in $H^1(\Gamma)$ perfectly confirms theoretical analysis. The optimal order $L^2$ error estimate was not covered by the theory. In experiments, we observe a somewhat less regular behaviour of $L^2(\Gamma)$ error for
the band width $d=h$. It becomes more regular if the bandwidth slightly growth, and for $d=5h$ we clearly see
the optimal second order of convergence. To compare results for two band widths in terms of accuracy versus computational costs, Table~\ref{tab1a} also shows the number of active degrees of freedom involved in computations in each case. Running experiments
with varying $\gamma$ (not shown), we concluded that taking any $\gamma\in [1,5]$ would be a reasonable choice, while increasing the band width further does not pay off in terms of  accuracy versus CPU time.
Results for  $\bH_h=0$ are very much similar to the 'exact' choice $\bH_h=\bH$. We note that this would not be the case if the band width $d$ is chosen to be $h$ independent.
In this case, setting $\bH_h=0$ leads to suboptimal convergence rates.

\begin{table}[ht]\small
\caption{Error norms and estimated convergence orders  in Experiment 1 with $\bH_h=\bH$.\label{tab1a}}
\begin{center}
\begin{tabular}{r|rrrrr|rrrrr}
\multicolumn{1}{c|}{} & \multicolumn{5}{c|}{$d=h$} & \multicolumn{5}{c}{$d=5h$} \\ \hline
 $\ell$    &\#d.o.f.& $L_2(\Gamma)$   &  & $H^1(\Gamma)$&  &\#d.o.f. & $L_2(\Gamma)$&  & $H^1(\Gamma)$&  \\ \hline
0 &   55 & 5.49e-2 &     & 1.25e-1&      &  138 & 6.19e-2&     & 1.27e+0& \\
1 &  107 & 1.35e-2 & 2.02& 6.22e-1& 1.01 &  274 & 1.51e-2& 2.03& 6.29e-1&  1.01\\
2 &  211 & 3.56e-3 & 1.92& 3.18e-1& 0.97 &  536 & 3.78e-3& 2.00& 3.20e-1&  0.98\\
3 &  411 & 8.88e-4 & 2.00& 1.60e-1& 0.99 & 1066 & 9.32e-4& 2.02& 1.60e-1&  1.00\\
4 &  834 & 2.34e-4 & 1.93& 7.95e-2& 1.00 & 2118 & 2.34e-4& 1.99& 7.98e-2&  1.00\\
5 & 1665 & 6.71e-5 & 1.80& 3.97e-2& 1.01 & 4251 & 5.88e-5& 2.00& 3.99e-2&  1.00\\
6 & 3329 & 1.27e-5 & 2.41& 1.98e-2& 1.00 & 8479 & 1.46e-5& 2.01& 1.99e-2&  1.00\\
7 & 6669 & 5.25e-6 & 1.27& 9.86e-3& 1.00 &16955 & 3.65e-6& 2.00& 9.90e-3&  1.01\\
8 &13319 & 2.82e-6 & 0.90& 4.89e-3& 1.00 &33911 & 8.82e-7& 2.05& 4.91e-3&  1.01\\
9 &26619 & 4.73e-7 & 2.57& 2.40e-3& 1.00 &67924 & 2.34e-7& 1.91& 2.41e-3&  1.03\\
10&53317 & 2.24e-7 & 1.08& 1.15e-3& 1.00 &135692& 5.41e-8& 2.11& 1.16e-3&  1.06\\
11&106630& 5.61e-8 & 1.99& 5.75e-4& 1.00 &271544& 1.34e-8& 2.02& 5.35e-4&  1.12\\ \hline
\end{tabular}
\end{center}
\end{table}

\begin{table}[ht]\small
\caption{Error norms and estimated convergence orders in Experiment 1 with $\bH_h=0$.\label{tab1b}}
\begin{center}
\begin{tabular}{r|rrrr|rrrr}
\multicolumn{1}{c|}{} & \multicolumn{4}{c|}{$d=h$} & \multicolumn{4}{c}{$d=5h$} \\ \hline
$\ell$ & $L_2(\Gamma)$   &  & $H^1(\Gamma)$&  & $L_2(\Gamma)$&  & $H^1(\Gamma)$&  \\ \hline
0  & 5.55e-2&      & 1.25e+0&      &2.90e-2&     &1.23e+0&\\
1  & 1.41e-2&  1.98& 6.22e-1& 1.00 &9.68e-3& 1.58&5.94e-1&  1.05\\
2  & 3.74e-3&  1.91& 3.18e-1& 0.97 &5.51e-3& 0.81&3.12e-1&  0.93\\
3  & 9.33e-4&  2.00& 1.60e-1& 1.00 &1.82e-3& 1.60&1.59e-1&  0.97\\
4  & 2.44e-4&  1.93& 7.95e-2& 1.01 &4.91e-4& 1.89&7.96e-2&  1.00\\
5  & 6.94e-5&  1.82& 3.97e-2& 1.00 &1.25e-4& 1.97&3.98e-2&  1.00\\
6  & 1.33e-5&  2.39& 1.98e-2& 1.00 &3.14e-5& 2.00&1.99e-2&  1.00\\
7  & 5.35e-6&  1.31& 9.86e-3& 1.01 &7.85e-6& 2.00&9.90e-3&  1.01\\
8  & 2.83e-6&  0.92& 4.89e-3& 1.01 &1.93e-6& 2.02&4.91e-3&  1.01\\
9  & 4.75e-7&  2.57& 2.40e-3& 1.03 &4.93e-7& 1.97&2.41e-3&  1.03\\
10 & 2.24e-7&  1.09& 1.15e-3& 1.06 &1.19e-7& 2.05&1.16e-3&  1.06\\
11 & 5.60e-8&  2.00& 5.75e-4& 1.01 &3.06e-8& 1.97&5.35e-4&  1.12\\ \hline
\end{tabular}
\end{center}
\end{table}

\medskip
\noindent{\it Experiment~2}. The second experiment is still for a 2D problem, but now
we test the method for a PDE posed on a surface with boundary. This case was not covered by the theory
in this paper. Let $\Gamma$ be a part of the curve $y=\sqrt{x}$ for $s\in(0,2)$, where $s$ is the arc length of  $\Gamma$ from the origin.  We are looking for the solution to the problem
\[
-\Delta_{\Gamma} u + u = f\quad \text{on}~ \Gamma,\qquad
u'(0) = u'(2) = 0
\]
The right hand side $f(s)$ is taken such that the exact solution is $u(s)=\cos(4 \pi s)$.

In this experiment, $\phi_h$ was used to define the narrow band $\Omega_h$ in \eqref{Omega_h}. The approximate signed distance function was computed  using the Matlab implementation of the closest point method  from \cite{Maurer}, which also gives the approximate projection $\bp$ needed to find the extension $f^e$ on $\Omega_h$. The extended problem uses approximate Hessian matrix recover from the distance function.

The unfitted finite element scheme is slightly modified to allow for Neumann conditions on both part of the boundary $\{(x, y): \|\phi(x,y)\| =kh\}$ and $\{(x, y) : s = 0\ or\ s=2\}$, and to handle the end points of $\Gamma$, as shown in Figure~\ref{fig1a}(a). $H^1(\Gamma)$ and $L^2(\Gamma)$ error norms for this experiment are shown in Table~\ref{tab2}. The optimal convergence order is clearly seen in the energy norm.
In the $L^2(\Gamma)$  norm the convergence pattern is slightly less regular, but close to the optimal order as well. Same conclusions hold if we set $\bH_h=0$ in $\Omega_h$.

\begin{table}[htbp]
\small
\caption{Error norms and estimated convergence orders  in Experiment 2 with $d=3h$.\label{tab2}}
\begin{center}
\begin{tabular}{r|rrrr|rrrr}
  & \multicolumn{4}{c|}{$\bH_h\approx\bH$}  & \multicolumn{4}{c}{$\bH_h=0$}  \\ \hline
$\ell$ & $L_2(\Gamma)$ &  & $H^1(\Gamma)$ &  & $L_2(\Gamma)$ &  &$H^1(\Gamma)$ & \\ \hline
0  & 8.54e-4 &      & 2.70e-2 &      & 8.95e-4 &      & 2.66e-2 & \\
1  & 2.22e-4 & 1.95 & 1.18e-2 & 1.19 & 2.10e-4 & 2.09 & 1.16e-2 & 1.20 \\
2  & 7.13e-5 & 1.64 & 5.69e-3 & 1.06 & 9.73e-5 & 1.11 & 5.42e-3 & 1.10 \\
3  & 1.64e-5 & 2.12 & 2.82e-3 & 1.01 & 2.01e-5 & 2.28 & 2.66e-3 & 1.03 \\
4  & 3.67e-6 & 2.16 & 1.41e-3 & 1.00 & 4.60e-6 & 2.13 & 1.33e-3 & 1.00 \\
5  & 1.04e-6 & 1.82 & 7.01e-4 & 1.01 & 1.69e-6 & 1.44 & 6.63e-4 & 1.00 \\
6  & 8.14e-7 & 0.35 & 3.48e-4 & 1.01 & 1.55e-6 & 0.13 & 3.28e-4 & 1.02 \\
7  & 1.36e-7 & 2.58 & 1.72e-4 & 1.02 & 2.80e-7 & 2.47 & 1.62e-4 & 1.02 \\
8  & 1.53e-8 & 3.15 & 8.38e-5 & 1.04 & 3.67e-8 & 2.93 & 7.88e-5 & 1.04 \\ \hline
 \end{tabular}
\end{center}
\end{table}

\begin{figure} [!ht]
\centering
\includegraphics[width=.45\textwidth]{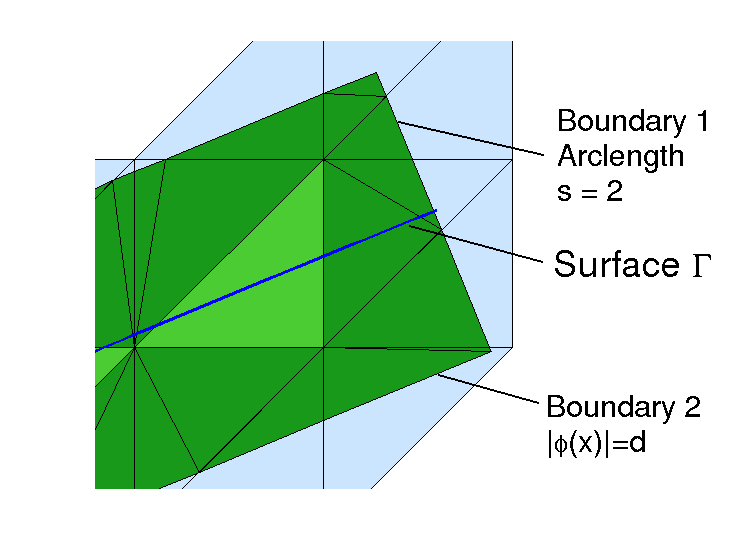}\qquad
\includegraphics[width=.47\textwidth]{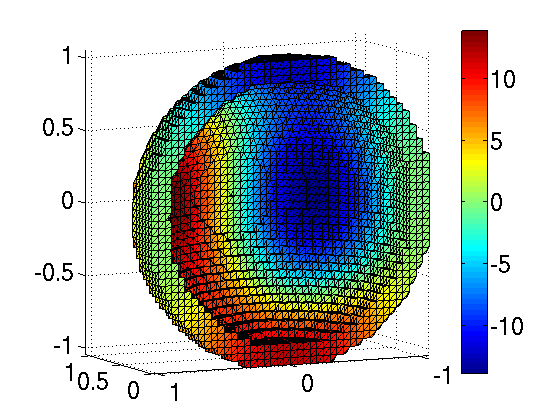}
\caption{Left: Unfitted narrow-band mesh for a surface with boundary in Experiment~2  \label{fig1a}; Right: Cutaway of a narrow band domain and numerical solution (full active tetrahedra from $\cT_d$ are shown, while integration is performed over cut tetrahedra, i.e. over $\Omega_h$). \label{overflow2}}
\end{figure}

\medskip
\noindent{\it Experiment~3}.
As the next test problem,
 we consider the Laplace--Beltrami equation \eqref{LBeq} on the unit sphere,
 $\Gamma=\{ \bx\in\R^3\mid \|\bx\|_2 = 1\}$.
The source term $f$ is taken such that the solution is given by
\[
    u(\bx)= \frac{12}{\|\bx\|^3}\left(3x_1^2x_2 - x_2^3\right),\quad
    \bx=(x_1,x_2,x_3)\in\Omega.
\]
Note that $u$ and $f$ are constant along normals at $\Gamma$.

We perform a regular  uniform tetrahedra subdivision  of $\Omega=(-2,2)^3$ with $h=2^{1-\ell}\times 10^{-1}$. Thus the grid is not aligned with $\dO_d$. We further refine only those elements which have non-empty intersection with  $\Omega_d$.
As before, we use piecewise affine continuous finite elements.
Optimal  convergence rates in $H^1(\Gamma)$ and $L^2(\Gamma)$ norms are observed with the narrow band width $d=h$ both with the exact choice  of $\bH_h=\bH$ and $\bH_h=0$, see Table~\ref{tab3}. The cutaway of $\Omega_h$ and computed solution with $d=h$ are illustrated in Figure~\ref{overflow2}.

\begin{table}[htbp]
\small
\caption{Error norms and estimated convergence orders  in Experiment 3 with $d=h$.\label{tab3}}
\begin{center}
\begin{tabular}{cr|rrrr|rrrr}
\multicolumn{2}{c|}{} & \multicolumn{4}{c|}{$\bH_h=\bH$}  & \multicolumn{4}{c}{$\bH_h=0$}  \\ \hline
 $\ell$    &\#d.o.f.&  $L_2(\Gamma)$ &  & $H^1(\Gamma)$ &  & $L_2(\Gamma)$ &  &$H^1(\Gamma)$ & \\ \hline
0&   1432 & 4.94e-1 &      & 2.78e+0 &     & 6.40e-1 &      & 3.01e+0 & \\
1&   5474 & 1.39e-1 & 1.83 & 5.67e-1 & 2.29& 1.64e-1 & 1.96 & 5.55e-1 & 2.44\\
2&  22084 & 3.62e-2 & 1.94 & 1.89e-1 & 1.58& 4.23e-2 & 1.96 & 1.87e-1 & 1.57\\
3&  88122 & 8.98e-3 & 2.01 & 8.08e-2 & 1.22& 1.05e-2 & 2.00 & 8.09e-2 & 1.21\\
4& 353920 & 2.35e-3 & 1.93 & 3.85e-2 & 1.07& 2.74e-3 & 1.94 & 3.86e-2 & 1.07\\
5&1416810 & 5.91e-4 & 1.99 & 1.90e-2 & 1.02& 6.93e-4 & 1.98 & 1.90e-2 & 1.02\\ \hline
\end{tabular}
\end{center}
\end{table}

\medskip
\noindent{\it Experiment~4}.
We repeat the previous experiment, but now for the equation posed on a torus instead of the unit sphere.
Let $\Gamma = \{ \bx\in\Omega \mid r^2 = x_3^2 + (\sqrt{x_1^2 + x_2^2} - R)^2\}$.
We take $R= 1$ and $r= 0.6$. In the coordinate system $(\rho, \phi,%
\theta)$, with
\[
    \bx = R\begin{pmatrix}\cos\phi \\ \sin\phi \\ 0 \end{pmatrix}
    + \rho\begin{pmatrix}\cos\phi\cos\theta \\ \sin\phi\cos\theta \\
    \sin\theta \end{pmatrix},
\]
the $\rho$-direction is normal to $\Gamma$, $\frac{\partial \bx}{\partial\rho}\perp\Gamma$ for $\bx\in\Gamma$.
The following solution $u$ and the corresponding right-hand side $f$ are constant in the normal direction:
\begin{equation}\label{polar2}
  \begin{split}
    u(\bx)&= \sin(3\phi)\cos( 3\theta + \phi),\\
    f(\bx)&= r^{-2} (9\sin( 3\phi)\cos( 3\theta + \phi)) \\
          & \quad + (R + r\cos( \theta))^{-2}(10\sin( 3\phi)\cos(3\theta + \phi) + 6\cos( 3\phi)\sin( 3\theta + \phi)) \\
          & \quad - (r(R + r\cos( \theta)))^{-1}(3\sin( \theta)\sin( 3\phi)\sin( 3\theta + \phi)) + u(\bx).
  \end{split}
\end{equation}

Near optimal convergence rates in $H^1(\Gamma)$ and $L^2(\Gamma)$ norms are observed with the narrow band width $d=h$, both with the exact choice  of $\bH_h=\bH$ and $\bH_h=0$. The surface norms of approximation errors for the example of torus are given in Table~\ref{tab4}. The solution is visualized in Figure~\ref{fig_torus}.

\begin{table}[htbp]
\small
\caption{Error norms and estimated convergence orders  in Experiment 4 with $d=h$.\label{tab4}}
\begin{center}
\begin{tabular}{cr|rrrr|rrrr}
& &\multicolumn{4}{c|}{$\bH_h=\bH$}  & \multicolumn{4}{c}{$\bH_h=0$}  \\ \hline
$\ell$& \# d.o.f.   & $L_2(\Gamma)$&     & $H^1(\Gamma)$&   & $L_2(\Gamma)$ &     &$H^1(\Gamma)$ &    \\ \hline
1 & 10094  &7.16e-2       &     & 3.03e+0     &    & 7.64e-2 &      & 4.29e+0 &       \\
2 & 41018  &1.93e-2       & 1.89& 1.35e+0     &1.17& 2.04e-2 & 1.91 & 1.60e+0 & 1.42  \\
3 & 165244 &4.95e-3       & 1.96& 6.30e-1     &1.10& 5.28e-3 & 1.95 & 6.63e-1 & 1.27  \\
4 & 664090 &1.26e-3       & 1.98& 2.96e-1     &1.09& 1.31e-3 & 2.01 & 3.04e-1 & 1.12  \\
5 & 2656782& 3.06e-4      & 2.04& 1.48e-1     &1.00& 3.23e-4 & 2.02 & 1.48e-1 & 1.04  \\ \hline
\end{tabular}
\end{center}
\end{table}

\begin{figure} [!ht]
\label{fig_torus}
\centering
\includegraphics[width=0.7\textwidth]{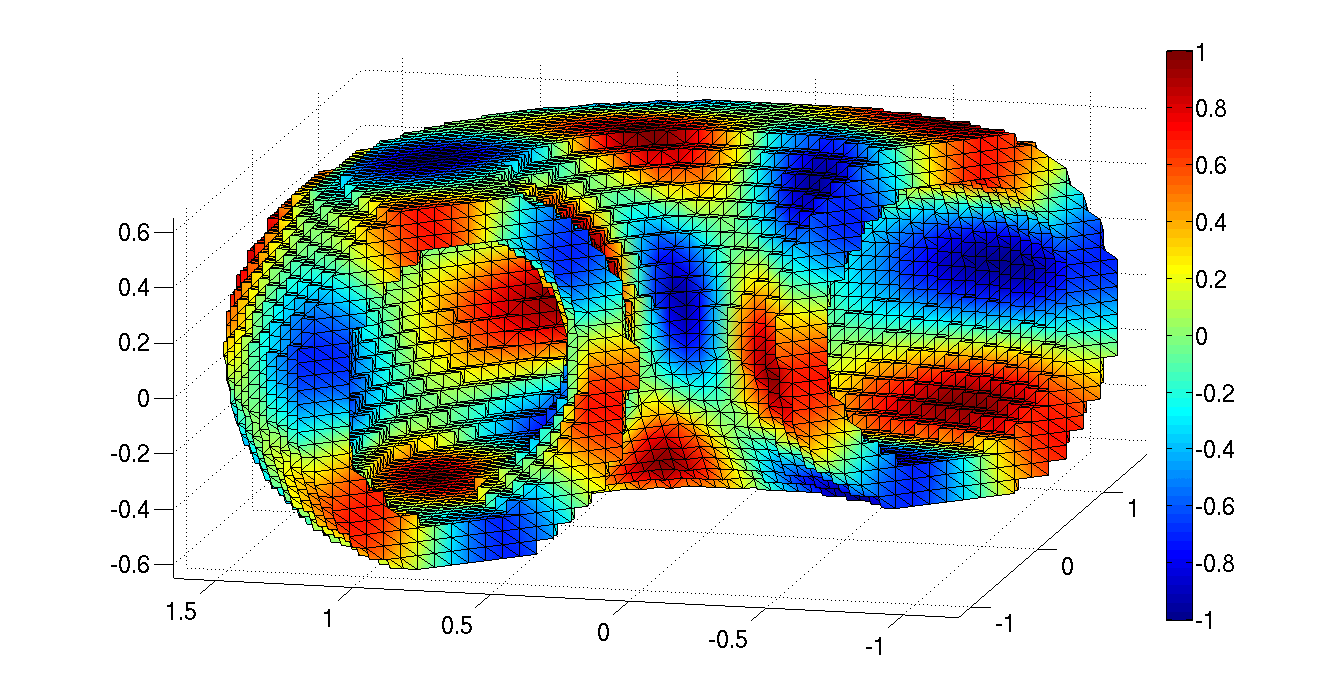}\qquad
\caption{Cutaway of a narrow band domain  and numerical solution for experiment 4 (full active tetrahedra are shown, while integration is performed over cut tetrahedra).}
\end{figure}

\medskip

\noindent{\it Experiment~5}.  Finally, we perform a few experiments with higher order finite element approximations. According to the analysis of section~\ref{s_error}, to achieve the optimal error convergence, one has to reconstruct $\phi$
and $\bH$ with higher accuracy, cf. \eqref{phi_h}, \eqref{Hh}. For example, $P_2$ elements call for the
approximation of $\phi$ with a piecewise second order polynomial function $\phi_h$, while a first order
reconstruction of $\bH$ is sufficient. The main technical difficulty here is implementing a sufficiently accurate
and cost efficient  numerical integration over implicitly given curved triangles, resulting from the intersection of $\Omega_h$
with a bulk triangulation.
In this paper, we get around this difficulty by constructing local (inside each cut triangle)
piecewise linear approximation of $\phi_h$ with some $h'=O(h^q)$, with $q$ from \eqref{phi_h}, i.e.
$h'=O(h^2)$ for $P_2$ elements and  $h'=O(h^3)$ for $P_3$ elements.  Further, an integral over a cut element is computed as 
a sum of integrals over the resulting set of smaller triangles;
\begin{wrapfigure}{l}{0.5\textwidth}
\label{fig:P2}
\centering
\includegraphics[width=0.5\textwidth]{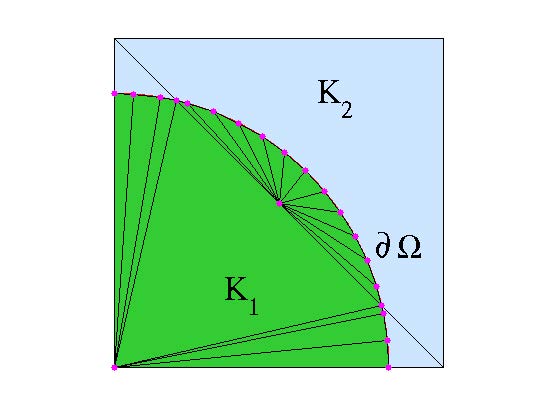}\qquad
\caption{Local subdivisions of cut triangles for the numerical integration of $P_2$ elements.}
\end{wrapfigure}
 see Figure~\ref{fig:P2} for the illustration of how such local triangulations were constructed for two cut triangles $K_1$ and $K_2$ of a bulk triangulation (FE functions are integrated over the green area).
The approach is suboptimal with respect to how the computational cost of building stiffness matrices scales if $h\to0$. We use this approach to illustrate the error analysis of this paper. The results of numerical experiments for the same problem as in Experiment~1 are shown in Table~\ref{table7}. Here we set $d=3h$, $\bH_h=\bH$ and observe optimal convergence order for $P_2$ and $P_3$ finite elements (the convergence stagnation for fine levels with $P_3$ elements is due to the influence of rounding off errors in our implementation). Unlike the case of $P_1$ finite elements, setting $\bH_h=0$ led to suboptimal convergence rates. This is consistent with the analysis.

\begin{table}[ht]\small
\caption{Error norms and estimated convergence orders for higher order elements in Experiment 5.\label{table7}}
\begin{center}
\begin{tabular}{r|rrrrr|rrrrr}
\multicolumn{1}{c|}{} & \multicolumn{5}{c|}{$P_2$ FE} & \multicolumn{5}{c}{$P_3$ FE} \\ \hline
 $\ell$    &\#d.o.f.& $L_2(\Gamma)$   &  & $H^1(\Gamma)$&  &\#d.o.f. & $L_2(\Gamma)$&  & $H^1(\Gamma)$&  \\ \hline
0 & 181  & 1.30e-3&     &7.36e-2 &      &  379  & 8.47e-5 &     & 7.13e-3 &     \\
1 & 357  & 1.75e-4&2.89 &2.03e-2 & 1.86 &  751  & 5.31e-6 & 4.00& 8.51e-4 &3.07 \\
2 & 709  & 2.31e-5&2.92 &5.29e-3 & 1.94 &  1495 & 3.24e-7 & 4.03& 1.04e-4 &3.03 \\
3 & 1381 & 2.96e-6&2.96 &1.35e-3 & 1.97 & 2911  & 1.95e-8 & 4.06& 1.27e-5 &3.03 \\
4 & 2817 & 3.73e-7&2.99 &3.38e-4 & 2.00 & 5950  & 2.98e-9 & 2.71& 1.55e-6 &3.03 \\
5 & 5629 & 4.69e-8&2.99 &8.48e-5 & 1.99 & 11893 & 2.55e-8 &-3.10& 2.08e-7 &2.90  \\
6 & 11261& 5.90e-9&2.99 &2.14e-5 & 1.99 & && & &  \\
\hline
\end{tabular}
\end{center}
\end{table}

\section{Conclusions} \label{s_concl}

We studied a formulation and a finite element method for elliptic partial differential equation
posed on hypersurfaces in $\mathbb{R}^N$, $N=2,3$. The formulation uses an extension of the equation off the surface to a volume domain containing the surface. The extension introduced in the paper results in uniformly elliptic problems in the volume domain.  This enables a straightforward application of standard discretization techniques, including higher order finite element methods. The method can be applied in a narrow band (although this is not a necessary requirement) and can be used with meshes not fitted to surface or computational domain boundary.  Numerical analysis reveals the sufficient conditions for the method to have  optimal convergence
order in the energy norm.
For $P_1$ finite elements and $h$-narrow band, the optimal convergence is achieved for a particular simple formulation. For higher order elements, an optimal complexity efficient implementation of the method is a subject of current research and will be reported in a follow-up paper.


\end{document}